\documentclass[preprint,a4paper,11pt,sort]{elsarticle}
\usepackage{fullpage}

\usepackage{amsmath,amssymb,mathtools} 
\usepackage{graphicx}

\usepackage{amsthm}
\usepackage{hyperref,xcolor}
\usepackage[capitalise,nameinlink]{cleveref}

\newtheorem{theorem}{Theorem}[section]
\newtheorem{lemma}[theorem]{Lemma}

\newtheorem{observation}[theorem]{Observation}

\theoremstyle{remark}
\newtheorem{remark}[theorem]{Remark}
\newcommand{\figdir}{./fig}

\DeclareMathOperator{\tw}{\mathsf{tw}}

\newcommand{\quadform}[2]{#2^{*} #1 #2}

\begin{document}

\title{An improved spectral lower bound of treewidth\tnoteref{t1}}
\tnotetext[t1]{%
Partially supported
by JSPS KAKENHI Grant Numbers 
JP20H05793, % Otachi (Ito Henkaku B)
JP20H05967, % Ono (Makino Henkaku A)
JP21K11752, % Otachi (Otachi C)
JP21H05852, % Hanaka (Minato Henkaku A)
JP21K17707, % Hanaka (Wakate)
JP21K19765, % Ono (Hoga)
JP22H00513. % Hanaka, Ono, Otachi (Ono A)
}

\author[1]{Tatsuya Gima}
\ead{gima@ist.hokudai.ac.jp}

\author[2]{Tesshu Hanaka}
\ead{hanaka@inf.kyushu-u.ac.jp}

\author[3]{Kohei Noro\fnref{fn-noro}}
\ead{noro.kohei@nagoya-u.jp}

\author[3]{Hirotaka Ono}
\ead{ono@nagoya-u.jp}

\author[3]{Yota Otachi\corref{cor1}}
\ead{otachi@nagoya-u.jp}

\cortext[cor1]{Corresponding author}
\fntext[fn-noro]{This work was done while he was a student at Nagoya University.
He is currently working at KOEI TECMO GAMES CO., LTD.}

\address[1]{Hokkaido University, Sapporo, Japan}
\address[2]{Kyushu University, Fukuoka, Japan}
\address[3]{Nagoya University, Nagoya, Japan}

\begin{abstract}
We show that for every $n$-vertex graph with at least one edge, 
its treewidth is greater than or equal to $n \lambda_{2} / (\Delta + \lambda_{2}) - 1$, where $\Delta$ and $\lambda_{2}$ are
the maximum degree and the second smallest Laplacian eigenvalue of the graph, respectively.
This lower bound improves the one by Chandran and Subramanian~[\textit{Inf.\ Process.\ Lett.}, 2003]
and the subsequent one by the authors of the present paper~[\textit{IEICE Trans.\ Inf.\ Syst.}, 2024].
The new lower bound is \emph{almost} tight in the sense that
there is an infinite family of graphs such that the lower bound is only~$1$ less than the treewidth
for each graph in the family.
Additionally, using similar techniques, 
we also present a lower bound of treewidth in terms of the largest and the second smallest Laplacian eigenvalues.
\end{abstract}


\maketitle

%%%%%%%%%%%%%%%%%%%%%%%%%%%%%%%%%%%%%%%%%%%%%%%%%%%%%%%%%%%%%%%%%%%%%%%%%%%%%%%%%%%%%%%%%%%%%%%%%%%%%%%%%%%%%%%%%%%%%%%%%%%%%%%%
%%%%%%%%%%%%%%%%%%%%%%%%%%%%%%%%%%%%%%%%%%%%%%%%%%%%%%%%%%%%%%%%%%%%%%%%%%%%%%%%%%%%%%%%%%%%%%%%%%%%%%%%%%%%%%%%%%%%%%%%%%%%%%%%
%%%%%%%%%%%%%%%%%%%%%%%%%%%%%%%%%%%%%%%%%%%%%%%%%%%%%%%%%%%%%%%%%%%%%%%%%%%%%%%%%%%%%%%%%%%%%%%%%%%%%%%%%%%%%%%%%%%%%%%%%%%%%%%%

\section{Introduction}

The concept of treewidth is one of the most well-studied measures of decomposability of graphs.
Roughly speaking, a graph with small treewidth admits a tree-like decomposition of small width,
which can be used in various ways, most notably for designing efficient algorithms~\cite{ArnborgLS91,Bodlaender98}.
Since the computation and the estimation of treewidth are important tasks,
several techniques for lower-bounding and upper-bounding treewidth are developed (see \cite{BodlaenderK10,BodlaenderK11}).
The most relevant to this paper is the spectral approach initiated by Chandran and Subramanian~\cite{ChandranS03}.
We follow this direction and present an improved lower bound of treewidth using the spectral approach.

Chandran and Subramanian~\cite{ChandranS03}
showed that for every $n$-vertex graph with at least one edge,
\begin{equation}
  \tw(G) \ge \frac{3 n \lambda_{2}}{4\Delta + 8 \lambda_{2}} - 1,
  \label{eq:ChandranS03}
\end{equation}
where $\tw(G)$ is the treewidth of $G$,
$\lambda_{2}$ is the second smallest Laplacian eigenvalue of $G$,
and $\Delta$ is the maximum degree of $G$.\footnote{See \cref{sec:pre} for definitions of the terms undefined here.}
Subsequently, the authors of the present paper showed a slightly improved lower bound~\cite{GimaHNOO2024}
in which the denominator $4\Delta + 8 \lambda_{2}$ in \cref{eq:ChandranS03} is replaced with 
$\max\{4\Delta + 3 \lambda_{2}, \, 3\Delta + 4.5 \lambda_{2}\}$.

In this paper, we further pursue lower bounds of treewidth in terms of the Laplacian eigenvalues
and obtain the following bound:
\begin{equation}
  \tw(G) \ge \frac{n \lambda_{2}}{\Delta + \lambda_{2}} -1.
  \label{eq:ours}
\end{equation}
As we will see later, the new bound is tight up to an additive factor of~$1$.
(See \cref{sec:improved-bound}.)

By using some of the techniques used for obtaining \cref{eq:ours}
and an inequality found by Gu and Liu~\cite{GuL22}, we also show the following lower bound, which is incomparable to \cref{eq:ours}:
\begin{equation}
  \tw(G) \ge \frac{2n \lambda_{2}}{3\lambda_{n} - \lambda_{2}}-1,
  \label{eq:ours-2}
\end{equation}
where $\lambda_{n}$ is the largest Laplacian eigenvalue of $G$.
This bound is tight for complete graphs.
(See \cref{sec:second-bound}.)

%%%%%%%%%%%%%%%%%%%%%%%%%%%%%%%%%%%%%%%%%%%%%%%%%%%%%%%%%%%%%%%%%%%%%%%%%%%%%%%%%%%%%%%%%%%%%%%%%%%%%%%%%%%%%%%%%%%%%%%%%%%%%%%%
%%%%%%%%%%%%%%%%%%%%%%%%%%%%%%%%%%%%%%%%%%%%%%%%%%%%%%%%%%%%%%%%%%%%%%%%%%%%%%%%%%%%%%%%%%%%%%%%%%%%%%%%%%%%%%%%%%%%%%%%%%%%%%%%
%%%%%%%%%%%%%%%%%%%%%%%%%%%%%%%%%%%%%%%%%%%%%%%%%%%%%%%%%%%%%%%%%%%%%%%%%%%%%%%%%%%%%%%%%%%%%%%%%%%%%%%%%%%%%%%%%%%%%%%%%%%%%%%%

\section{Preliminaries}
\label{sec:pre}

\subsection*{Laplacian eigenvalues}
Let $G = (V,E)$ be an $n$-vertex graph. The \emph{Laplacian matrix} of $G$ is the $n \times n$ matrix $L \in \mathbb{Z}^{V \times V}$ such that
\[
  L(u,v) = 
  \begin{cases}
    \deg(u) & u = v, \\
    \ -1 & \{u,v\} \in E, \\
    \ \phantom{+}0 & \text{otherwise}.
  \end{cases}
\]
The \emph{Laplacian eigenvalues} of $G$ are the eigenvalues of the Laplacian matrix of $G$.
We often denote the Laplacian eigenvalues of a graph by $\lambda_{1}, \lambda_{2}, \dots, \lambda_{n}$ in non-decreasing order.

It is known that $\lambda_{1} = 0$ for every graph.
The second smallest Laplacian eigenvalue $\lambda_{2}$ is also known as the \emph{algebraic connectivity}~\cite{Fiedler73} or the \emph{Fiedler value}~\cite{SpielmanT07}.
It is known that $\lambda_{2}$ is bounded from above by the vertex connectivity of the graph (see~\cite{BrouwerH2011}).
For the largest Laplacian eigenvalue $\lambda_{n}$, it is known that for every graph with at least one edge,
it holds that $\Delta + 1 \le \lambda_{n} \le \max_{\{u,v\} \in E} (\deg(u) + \deg(v))$~\cite{GroneM94,AndersonM71}.
It is known that $\lambda_{2} = \lambda_{n}$ if and only if
the graph is either edge-less or complete~(see \cite[Propositions~1.3.3 and 1.3.7, and Section~1.4.1]{BrouwerH2011}).

\subsection*{Treewidth}
Although the treewidth is the central concept in this paper, we do not need its precise definition in our proofs.
We only need the following property (\cref{lem:RobertsonS86}) shown by Robertson and Seymour~\cite{RobertsonS86}.
(See e.g., \cite{Bodlaender98} for the definition of treewidth.)

For a graph $G = (V,E)$, we denote its \emph{treewidth} by $\tw(G)$.
For $S \subseteq V$, let $G-S$ denote the subgraph of $G$ obtained by deleting the vertices in $S$.
\begin{lemma}
[\cite{RobertsonS86}]
\label{lem:RobertsonS86}
Every $n$-vertex graph $G = (V,E)$ has a vertex set $S \subseteq V$ with $|S| \le \tw(G)+1$
such that each connected component of $G - S$ contains at most $\frac{1}{2} (n-|S|)$ vertices.
\end{lemma}
\cref{lem:RobertsonS86} was used in the previous studies~\cite{ChandranS03,GimaHNOO2024} as well.
One of the main technical differences in this study from the previous ones is that 
we use \cref{lem:RobertsonS86} in the following form,
which can be obtained by appropriately partitioning the connected components of $G - S$ into (at most) three sets.
\begin{lemma}
[See e.g., {\cite[Lemma 9]{GruberH08}}] % ~\cite{GruberH08} % Lemma 9  %  (see also \cite{BeckerG96}). % Lemma 1
\label{lem:balanced-three-partition}
For every graph $G = (V,E)$, there is a set $S \subseteq V$ with $|S| \le \tw(G)+1$ such that
$V \setminus S$ can be partitioned into three (possibly empty) sets $U_{1}$, $U_{2}$, $U_{3}$ with the following properties:
\begin{itemize}
  \item $|U_{i}| \le \frac{1}{2} (n-|S|)$ for each $i \in \{1,2,3\}$; and
  \item there is no edge between two sets $U_{i}$ and $U_{j}$ for $i \ne j$.
\end{itemize}
\end{lemma}

%%%%%%%%%%%%%%%%%%%%%%%%%%%%%%%%%%%%%%%%%%%%%%%%%%%%%%%%%%%%%%%%%%%%%%%%%%%%%%%%%%%%%%%%%%%%%%%%%%%%%%%%%%%%%%%%%%%%%%%%%%%%%%%%
%%%%%%%%%%%%%%%%%%%%%%%%%%%%%%%%%%%%%%%%%%%%%%%%%%%%%%%%%%%%%%%%%%%%%%%%%%%%%%%%%%%%%%%%%%%%%%%%%%%%%%%%%%%%%%%%%%%%%%%%%%%%%%%%
%%%%%%%%%%%%%%%%%%%%%%%%%%%%%%%%%%%%%%%%%%%%%%%%%%%%%%%%%%%%%%%%%%%%%%%%%%%%%%%%%%%%%%%%%%%%%%%%%%%%%%%%%%%%%%%%%%%%%%%%%%%%%%%%

\section{The improved lower bound}
\label{sec:improved-bound}

To prove the improved lower bound (\cref{eq:ours}), we need \cref{lem:L-quadratic-form,lem:Qff} below, which are well known.
\cref{lem:L-quadratic-form} is easy to derive from the definition of the Laplacian matrix and 
often used for real vectors $x$ (see e.g., \cite{AlonM85}).
\cref{lem:Qff} holds by Rayleigh's principle (see \cite{NobleD1977}).
For a complex vector $x$, we denote its conjugate transpose by $x^{*}$.
\begin{lemma}
\label{lem:L-quadratic-form} % the Laplacian quadratic form
Let $L$ be the Laplacian matrix of a graph $G = (V,E)$.
For every $x \in \mathbb{C}^{V}$, it holds that
$\quadform{L}{x} = \sum_{\{u,v\} \in E} |x(u) - x(v)|^{2}$.
\end{lemma}

\begin{lemma}
\label{lem:Qff}
Let $\lambda_{2}$ be the second smallest Laplacian eigenvalue of a graph $G = (V,E)$.
If $x \in \mathbb{C}^{V}$ satisfies $\sum_{v \in V} x(v) = 0$, then $\quadform{L}{x} \ge \lambda_{2} \lVert x \rVert^2$.
\end{lemma}

The following simple observation allows us to find $x \in \mathbb{C}^{V}$ in \cref{lem:L-quadratic-form,lem:Qff}
that achieves a good lower bound of treewidth using \cref{lem:balanced-three-partition}.

\begin{observation}
\label{obs:triangle}
Let $a$, $b$, $c$ be nonnegative real numbers.
If $\max\{a,b,c\} \le (a+b+c)/2$, then
there exist complex numbers $\alpha$, $\beta$, $\gamma$ such that
$|\alpha| = |\beta| = |\gamma| = 1$ and
$a \alpha + b \beta + c \gamma = 0$.
\end{observation}
\begin{proof}
In this proof, $i$ denotes the imaginary unit $\sqrt{-1}$.
By symmetry, we may assume that $\max\{a,b,c\} = c$. 
Since $c \le (a+b+c)/2$, we have $a + b \ge c$.
If $a + b = c$, then we set 
$\alpha = 1$, $\beta = 1$, $\gamma = -1$
and get $a \alpha + b \beta + c \gamma = 0$ and $|\alpha| = |\beta| = |\gamma| = 1$.
Assume that $a + b > c$. Then, there exists a triangle $ABC$ of side-lengths $a = BC$, $b = CA$, $c = AB$.
Let $\theta = \angle{ABC}$ and $\phi = \angle{CAB}$.
In the complex plane, the traversal $B$--$C$--$A$--$B$ along the triangle can be represented 
by $a (\cos \theta + i \sin \theta)$, $b (\cos \phi - i \sin \phi)$, and $-c$ (see \cref{fig:triangle}).
Thus, by setting $\alpha = \cos\theta + i \sin\theta$, $\beta=\cos\phi - i \sin\phi$, and $\gamma = -1$,
we have $a \alpha + b \beta + c \gamma = 0$ and $|\alpha| = |\beta| = |\gamma| = 1$.
\end{proof}
\begin{figure}[htb]
  \centering
  \includegraphics[scale=1]{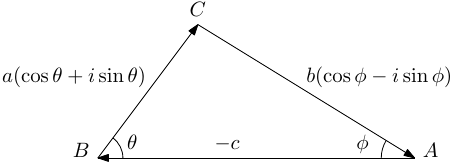}
  \caption{A visual proof of \cref{obs:triangle}.}
  \label{fig:triangle}
\end{figure}

Now we are ready to prove \cref{eq:ours}.
\begin{theorem}
\label{thm:tw-lb}
For every $n$-vertex graph $G$ with maximum degree $\Delta \ge 1$, it holds that
\[
  \tw(G) \ge \frac{n \lambda_{2} }{\Delta + \lambda_{2}} -1,
\]
where $\lambda_{2}$ is the second smallest Laplacian eigenvalue of $G$.
\end{theorem}

\begin{proof}
Let $G = (V,E)$.
By \cref{lem:balanced-three-partition}, there is a partition $(S, U_{1}, U_{2}, U_{3})$  of $V$
such that $|S| \le \tw(G)+1$, $|U_{i}| \le \frac{1}{2} (n-|S|)$ for each $i \in \{1,2,3\}$, and
there is no edge connecting two sets $U_{i}$ and $U_{j}$ for $i \ne j$, where some of the four sets may be empty.

By \cref{obs:triangle}, there exist complex numbers $\alpha_{1}$, $\alpha_{2}$, $\alpha_{3}$
such that $|\alpha_{i}| = 1$ for $i \in \{1, 2, 3\}$ and $\sum_{i \in \{1,2,3\}} |U_{i}| \, \alpha_{i} = 0$.
We define $x \in \mathbb{C}^{V}$ as follows:
\[
  x(v)=
  \begin{cases}
    \alpha_{i} & v \in U_{i}, \\
    0          & v \in S.
  \end{cases}
\]
Since $\sum_{v \in V} x(v) = \sum_{i \in \{1,2,3\}} |U_{i}| \, \alpha_{i} = 0$,
\cref{lem:Qff} implies that $\quadform{L}{x} \ge \lambda_{2} \lVert x \rVert^{2}$.
Since $\lVert x \rVert^2 = \sum_{i \in \{1,2,3\}} |\alpha_{i}|^{2} \, |U_{i}| = \sum_{i \in \{1,2,3\}} |U_{i}| =  n - |S|$,
we have 
\[
  \quadform{L}{x} \ge \lambda_{2} (n - |S|).
\]

By \cref{lem:L-quadratic-form}, $\quadform{L}{x} = \sum_{\{u,v\} \in E} |x(u) - x(v)|^{2}$ holds.
Observe that an edge connecting two vertices in the same part of the partition $(S, U_{1}, U_{2}, U_{3})$ does not contribute to this sum 
as its endpoints have the same value under $x$.
Since there is no edge between two sets $U_{i}$ and $U_{j}$ for $i \ne j$,
only the edges between $S$ and $V \setminus S$ contribute to the sum $\sum_{\{u,v\} \in E} |x(u) - x(v)|^{2}$.
For an edge $\{u,v\}$ with $u \in S$ and $v \in V \setminus S$,
we have $|x(u) - x(v)|^{2} = |x(v)|^{2} = 1$ as $x(v) = \alpha_{i}$ for some $i \in \{1,2,3\}$.
Since there are at most $\Delta |S|$ edges between $S$ and $V \setminus S$, it holds that
\[
  \quadform{L}{x} \le \Delta |S|.
\]

By combining the lower and upper bounds of $\quadform{L}{x}$ obtained above, we get $\Delta |S| \ge \lambda_{2} (n - |S|)$,
and thus $|S| \ge n \lambda_{2} / (\Delta + \lambda_{2})$.
As $\tw(G) \ge |S| -1$, the theorem follows.
\end{proof}

\begin{remark}
We can see that \cref{thm:tw-lb} is \emph{almost} tight for complete bipartite graphs as follows.
For positive integers $p$ and $q$ with $p \le q$, let $K_{p,q}$ denote the complete bipartite graph with $p$ vertices on one side and $q$ vertices on the other side.
The maximum degree $\Delta$ of $K_{p,q}$ is $\max\{p,q\} = q$.
It is known that the second smallest Laplacian eigenvalue $\lambda_{2}$ of $K_{p,q}$ is $\min\{p,q\} = p$~\cite[Section~1.4.2]{BrouwerH2011}.
Thus, \cref{thm:tw-lb} gives a lower bound of $(p+q) p / (q + p) -1 = p - 1$.
This lower bound is only~$1$ less than the exact value as $\tw(K_{p,q}) = \min\{p,q\} = p$~\cite{Bodlaender98}.
\end{remark}

\section{The lower bound in terms of $\lambda_{2}$ and $\lambda_{n}$}
\label{sec:second-bound}

Gu and Liu~\cite{GuL22} studied a spectral lower bound of the matching number
and showed the following lemma as one of their main tools.
\begin{lemma}
[\cite{GuL22}]
\label{lem:lambda2-and-n}
Let $G = (V,E)$ be an $n$-vertex non-complete graph with at least one edge.
Let $(S, X, Y)$ be a partition of $V$ such that % $G-S$ is disconnected,
$|X| \le |Y|$ and there is no edge between $X$ and $Y$.
Then,
\[
  |S| \ge \frac{2 \lambda_{2}}{\lambda_{n} - \lambda_{2}} \cdot |X|,
\]
where $\lambda_{2}$ and $\lambda_{n}$ are 
the second smallest and the largest Laplacian eigenvalues of $G$, respectively.
\end{lemma}

Our second lower bound (\cref{eq:ours-2}) can be shown by combining \cref{lem:lambda2-and-n} with \cref{lem:balanced-three-partition}.
\begin{theorem}
\label{thm:tw-lb2}
For every $n$-vertex graph $G$ with at least one edge, 
\[
  \tw(G) \ge \frac{2n \lambda_{2}}{3\lambda_{n} - \lambda_{2}}-1,
\]
where $\lambda_{2}$ and $\lambda_{n}$ are 
the second smallest and the largest Laplacian eigenvalues of $G$, respectively.
\end{theorem}
\begin{proof}
Let $G = (V,E)$. 
If $G$ is the $n$-vertex complete graph $K_{n}$, then $\lambda_{2} = \lambda_{n} = n$ holds~\cite[Section~1.4.1]{BrouwerH2011},
and thus $2n \lambda_{2}/(3\lambda_{n} - \lambda_{2})-1 = n-1 = \tw(K_{n})$.
In the following, we assume that $G$ is not a complete graph (but has at least one edge).

As in the proof of \cref{thm:tw-lb}, let $(S, U_{1}, U_{2}, U_{3})$ be a partition of $V$ that satisfies the conditions in \cref{lem:balanced-three-partition}; 
that is, $|S| \le \tw(G)+1$, $|U_{i}| \le \frac{1}{2} (n-|S|)$ for each $i \in \{1,2,3\}$, and
there is no edge connecting two sets $U_{i}$ and $U_{j}$ for $i \ne j$, 
where some of the four sets may be empty.

For each $i \in \{1, 2, 3\}$, as $|U_{i}| \le |(U_{1} \cup U_{2} \cup U_{3}) \setminus U_{i}| = |V \setminus (S \cup U_{i})|$,
we can apply \cref{lem:lambda2-and-n} by setting $X = U_{i}$ and $Y = V \setminus (S \cup U_{i})$, and we get
\[
  |S| \ge \frac{2\lambda_{2}}{\lambda_{n} - \lambda_{2}} \cdot |U_{i}|.
\]
By adding the inequalities for all $i \in \{1, 2, 3\}$, we obtain
\[
  3|S|
  \ge \frac{2\lambda_{2}}{\lambda_{n} - \lambda_{2}} \, (|U_{1}| + |U_{2}| + |U_{3}|)
  = \frac{2\lambda_{2}}{\lambda_{n} - \lambda_{2}} \, (n - |S|),
\]
which implies $|S| \ge 2n \lambda_{2} / (3\lambda_{n} - \lambda_{2})$.
Now the theorem follows as $\tw(G) \ge |S| -1$.
\end{proof}

\begin{remark}
\cref{thm:tw-lb2} is tight for complete graphs
as observed in the first paragraph of its proof.
\end{remark}

%%%%%%%%%%%%%%%%%%%%%%%%%%%%%%%%%%%%%%%%%%%%%%%%%%%%%%%%%%%%%%%%%%%%%%%%%%%%%%%%%%%%%%%%%%%%%%%%%%%%%%%%%%%%%%%%%%%%%%%%%%%%%%%%
%%%%%%%%%%%%%%%%%%%%%%%%%%%%%%%%%%%%%%%%%%%%%%%%%%%%%%%%%%%%%%%%%%%%%%%%%%%%%%%%%%%%%%%%%%%%%%%%%%%%%%%%%%%%%%%%%%%%%%%%%%%%%%%%
%%%%%%%%%%%%%%%%%%%%%%%%%%%%%%%%%%%%%%%%%%%%%%%%%%%%%%%%%%%%%%%%%%%%%%%%%%%%%%%%%%%%%%%%%%%%%%%%%%%%%%%%%%%%%%%%%%%%%%%%%%%%%%%%
%%%%%%%%%%%%%%%%%%%%%%%%%%%%%%%%%%%%%%%%%%%%%%%%%%%%%%%%%%%%%%%%%%%%%%%%%%%%%%%%%%%%%%%%%%%%%%%%%%%%%%%%%%%%%%%%%%%%%%%%%%%%%%%%

\section{Conclusion}

In this paper, we have shown an improved lower bound of treewidth in terms of the second smallest Laplacian eigenvalue (\cref{thm:tw-lb})
and another lower bound in terms of the largest and the second smallest Laplacian eigenvalues (\cref{thm:tw-lb2}).
The first one is almost tight for complete bipartite graphs with an additive gap of~$1$,
and the second one is tight for complete graphs.
It would be interesting to determine whether the first bound is already tight for some graphs or if it can be further improved.

%%%%%%%%%%%%%%%%%%%%%%%%%%%%%%%%%%%%%%%%%%%%%%%%%%%%%%%%%%%%%%%%%%%%%%%%%%%%%%%%%%%%%%%%%%%%%%%%%%%%%%%%%%%%%%%%%%%%%%%%%%%%%%%%

\section*{Acknowledgments}
The authors thank Nobutaka Shimizu for valuable comments and discussions.

%%%%%%%%%%%%%%%%%%%%%%%%%%%%%%%%%%%%%%%%%%%%%%%%%%%%%%%%%%%%%%%%%%%%%%%%%%%%%%%%%%%%%%%%%%%%%%%%%%%%%%%%%%%%%%%%%%%%%%%%%%%%%%%%
%%%%%%%%%%%%%%%%%%%%%%%%%%%%%%%%%%%%%%%%%%%%%%%%%%%%%%%%%%%%%%%%%%%%%%%%%%%%%%%%%%%%%%%%%%%%%%%%%%%%%%%%%%%%%%%%%%%%%%%%%%%%%%%%
%%%%%%%%%%%%%%%%%%%%%%%%%%%%%%%%%%%%%%%%%%%%%%%%%%%%%%%%%%%%%%%%%%%%%%%%%%%%%%%%%%%%%%%%%%%%%%%%%%%%%%%%%%%%%%%%%%%%%%%%%%%%%%%%

\bibliographystyle{plainurl}
\bibliography{sptw2}

\end{document}